\documentclass[reqno,12pt, a4paper]{amsart}
\usepackage{amsmath,amsfonts,amssymb}
\usepackage{verbatim}
\usepackage{enumerate}
\usepackage[left=2cm,top=2cm,right=2cm,bottom=2cm]{geometry}
\usepackage{url}
\usepackage{tikz}
\usepackage[all]{xy}
\def\N{\mathbb N}

\def\Z{\mathbb Z}

\def\Q {\mathbb Q}

\def\F{\mathbb F}

\def\ord{\mathop{\rm ord}\nolimits}

\def\lll{l}
\theoremstyle{plain}
\newtheorem{theorem}{Theorem}[section]
\newtheorem{lemma}[theorem]{Lemma}
\newtheorem{definition}[theorem]{Definition}
\newtheorem{corollary}[theorem]{Corollary}

\newtheorem{remark}[theorem]{Remark}

\def\qed{\hfill\hbox{$\square$}}

\theoremstyle{definition}

\numberwithin{equation}{section}

\author[F.E. Brochero Mart\'{\i}nez]{F. E. Brochero Mart\'{\i}nez}
\author[L. Batista de Oliveira]{L. Batista de Oliveira}
\author[C. R. Giraldo Vergara]{C. R. Giraldo Vergara}
\address{
Departamento de Matem\'{a}tica\\
Universidade Federal de Minas Gerais\\
UFMG\\
Belo Horizonte, MG\\
 30123-970\\
 Brazil\\
 }
 \email{fbrocher@mat.ufmg.br }\email{L. Batista  de Oliveira (in memoriam)}
\email{carmita@mat.ufmg.br}

\title[Wedderburn Decomposition and Idempotents \dots]{Wedderburn Decomposition and Idempotents of some finite metacyclic group algebras}

\begin{document}
\maketitle

\begin{abstract} 
In this article,  we  show  explicitly the  Wedderburn decomposition of the  metacyclic group algebra $\F_qG$, where $G$ has a cyclic subgroup of index 2 and  $\gcd(|G|,q)=1$.  We also construct  the complete set of central and left idempotents  of these group algebras.
\end{abstract}

\section{Introduction}

Let $\F_q$ be a finite field with $q$ elements, $G$ be a finite group with $n$ elements, where $\gcd (q,n )=1$,  and \(\F_qG\)  be the group algebra of \(G\) over  \(\F_q\).  Since $q$ and $n$ are coprimes, it follows from Maschke's Theorem  that  \(\F_qG\)   is semisimple and, as a consequence of the Wedderburn-Artin Theorem,   \(\F_qG\)  is isomorphic to a direct sum of matrix algebras over division rings. 
In addition, by Wedderburn's  Little Theorem, it is known that  finite  divisions rings are actually fields that are  in our case,  a finite extensions of $\F_q$, i.e.  there exists an isomorphism $\rho$ such that
$$\F_qG\mathop{\simeq}^\rho M_{l_1}(\F_{q^{m_1}})\oplus  M_{l_2}( \F_{q^{m_2}})\oplus\cdots\oplus M_{l_t} (\F_{q^{m_t}})$$
where $l_1,\dots l_t, m_1,\dots,m_t$ are appropriate positive integers  such that $\sum_{j=1}^t l_j^2m_j=|G|$.

 The explicit description of the primitive idempotents and Wedderburn decomposition of  \(\F_qG\)   is an important  problem in group algebras. In addition,  determining  ideals of \(\F_qG\) is important in coding theory, because these ideals  can be seen as subspaces of the vector field \(\F_q^n\)    that has  additional algebraic   properties. For instance,  irreducible cyclic codes are ideals of group algebra \(\F_q C_n\simeq \dfrac {\F_q[x]}{(x^n-1)}\)  generate by irreducible factors of \(x^n-1\).


Observe that $\F_qG$ has $t$  central irreducible idempotents, each one of the form 
$$e_i=\rho^{-1} (0\oplus\cdots\oplus 0\oplus I_i\oplus0\cdots\oplus 0),$$
where $I_i$ represents the identity matrix of the component $M_{l_i}(\F_{q^{m_i}})$.  Then, the isomorphism $\rho$ determines explicitly each central irreducible idempotent. 

In the case when we consider the field $\Q$ instead of $\F_q$, the calculus of  central idempotents and Wedderburn decomposition is widely studied;  the classical method to calculate  the primitive central idempotents of group algebras depends on   computing   the  character group table.  
Other method is shown in \cite{JLP}, where  Jespers, Leal and Paques describe the central irreducible idempotents, when $G$ is a nilpotent group, using the structure of its subgroups without 
employing the characters of the group. 
Generalizations and improvements of this method can be found in  \cite{ORS}, where the authors   provide information about the Wedderburn decomposition of $\Q G$. 
This computational method is also used in \cite{BrRi} to compute the Wedderburn decomposition and the central primitive
 idempotents of a  finite semisimple group algebra  $KG$, where $G$ is  an abelian-by-supersolvable group
$G$  and  $K$ is a finite field. 

The structure of $KG$ when $G=D_{2n}$ is the dihedral group with $2n$ elements is well known for $K=\Q$ (see \cite{BrGi2}) and for $K=\F_q$ where $\gcd(2n,q)=1$ (see \cite{Bro2}). 
In \cite{DFP}, Dutra, Ferraz and Polcino Milies impose conditions  over $q$ and $n$ in order for $\F_qD_{2n}$ to have the same  number of irreducible components as that of  $\Q D_{2n}$.  This result is generalized in \cite{FGPM}, where Ferraz, Goodaire and Polcino Milies find, for some families of groups,  conditions on $q$ and $G$ in order for $\F_qG$ to have the minimum number of simple components.

In this article, we  show  explicitly the  Wedderburn decomposition of the group algebra $\F_qG$ where $G$ has a  cyclic subgroup of index 2. We consider two possible group's families: The split metacyclic group that has a presentation of the form  $G = \langle{x,y \mid x^n= 1 = y^2 , xy =y  x^s \rangle}$ and the non-split metacyclic group that has a presentation of the form $G = \langle x,y \mid x^{2n}= 1, y^2 = x^n , xy = yx^s \rangle$.  In each case, we also find the complete set of central orthogonal idempotents and, in addition, we decompose  each central idempotent, when possible, in sum of non-central orthogonal idempotents.

\section{Preliminaries}
Throughout this article, $\F_q$ and $\overline\F_q$ denote a finite fields with $q$ elements and its  algebraic closure respectively,    $ord(a)$ is the order of $a$ in the cyclic group $\F_q^*$ and for each $c$ and $d$   positive integers  such that $\gcd(c,d)=1$, 
$ord_c(d)$ denotes the order of $d$ in the multiplicative group $\Z_c^*$.  For each $f(x)\in \F_q[x]$ such that $f(0)\ne 0$, $ord(f)$ is the less positive integer $n$ such that $f(x)$ divide $x^n-1$.  For each prime number $p$,  the function $\nu_p$ is the $p$-valuation, i.e. for each integer $c$, $\nu_p(c)$ is  the  highest exponent $v$ such that $p^v$ divides $c$.  In addition, for each  positive integer $k$, $\Phi_k(x)$ denotes the $k$-th cyclotomic polynomial.

The decomposition into simple components  of the group algebra  $\F_qG$,  when $G=C_n$ is a cyclic group with $n$ elements and $\gcd(n,q)=1$  is a well known result  and it can be seen  as direct consequence of the Chinese Remainder Theorem. Indeed 
\begin{equation}\F_qC_n\simeq \frac {\F_q[x]}{\langle x^n-1\rangle}\simeq \sum_{d|n} \frac {\F_q[x]}{\langle \Phi_d(x)\rangle}.\label{decom_ciclico}\end{equation}
In addition, since  $\Phi_d(x)$ splits into $\frac {\phi(d)}{\ord_d q}$ irreducible factors in $\F_q[x]$, each ring  $\frac {\F_q[x]}{\langle \Phi_d(x)\rangle}$ can be decompose  into $\frac {\phi(d)}{\ord_d q}$ copies of the field $\F_q(\xi_d)$, where $\xi_d$ is a root of $\Phi_d(x)$. In general, we have the following result for abelian group.


%

\begin{theorem}[Perlis-Walker's Theorem]\label{perliswalker}
Let  $G$ be an abelian group  with $n$ elements and $K$ be a field such that  $char(K) \nmid n$. Then 
$$ KG \cong \displaystyle\bigoplus_{d|n}a_dK(\xi_d),$$
where  $\xi_d$ is a $d$-th primitive root of unit,   $a_d = \frac{n_d}{[K(\xi_d):K]}$ and  $n_d$ is the number of elements of order $d$ in $G$.
\end{theorem}

We observe that the decomposition of $\F_qC_n$ depends fundamentally on the factorization of $x^n-1$ in $\F_q[x]$.  For the groups that we consider in this paper, that factorization will also be  essential as well as the $s$-self-involutive propriety. In fact,  for each polynomial $f(x)$, we construct  the so-called   $s$-involution  of $f(x)$, 
that  it is  one possible  generalization of  the notion of  reciprocal polynomial.

%
\begin{definition}Let $g(x) = (x-a_1)(x-a_2)\cdots (x-a_k)$ be a polynomial over $\F_q[x]$ and $s \in \Z$ such that $s^2\equiv1 \pmod {ord(g)}$. Let us denote $g^{*_s}(x)$,   the s-involution of $g(x)$, defined as  $g^{*_s}(x) := (x-a_1^s)(x-a_2^s)\cdots (x-a_k^s)$.   In the case when  $g$ and $g^{*_s}$ have the same roots in the decomposition field, the polynomial  $g$ is called $s$-self-involutive. 
The particular case when $s=-1$, the $s$-involution of $g(x)$  is the classical reciprocal of $g(x)$.
  \end{definition}


The following  result provides some properties on the $p$-adic valuation on numbers of the form $a^k-1$ with $a\equiv 1\pmod p$,  that we will using to determine the relation between  the finite fields $\F_q(\xi_d)$ and $\F_q(\xi_d+\xi_d^s, \xi_d^{s+1})$. This result is attributed to E. Lucas  and R. D. Carmichael  \cite{Car}.
\begin{lemma}\cite[Proposition~1]{B77} \label{LEL} Let $p$ be a prime and $\nu_p$ be the $p$-valuation. 
The following hold:
\begin{enumerate}[1)]
\item if  $p$ is an odd prime that divides $a-1$, then $\nu_{p}(a^{k}-1)=\nu_{p}(a-1)+\nu_{p}(k)$;
\item if $p=2$  and $a$ is odd number, then
$$
\nu_{2}(a^{k}-1)=
\begin{cases}
\nu_{2}(a-1)&\text{if $k$ is odd,} \\
\nu_{2}(a^{2}-1)+\nu_{2}(k)- 1&\text{if $k$ is even.}\\
\end{cases}
$$
\end{enumerate}
\end{lemma}
The following lemma shows properties of some extensions field, that 
will be useful in one of  principal theorem of this paper.

\begin{lemma}\label{lema1}
Let $f(x) \in \F_q[x]$ be  a monic irreducible $s$-self-involutive factor   of $x^n-1$, where $s^2\equiv 1\pmod n$ and $\alpha$ be a root of $f(x)$, then
$$ [ \F_q(\alpha) : \F_q(\alpha +\alpha^s, \alpha^{s+1}) ] = 2. $$
\end{lemma}

\begin{proof}
Since $f$ is an irreducible polynomial, then $f \in \F_q[x]$ is the minimal  polynomial of  $\alpha$ over $\F_q$. Therefore $\alpha  \in \F_{q^m}$, where $m$ is the degree of $f$ and $m$  is minimal  with this property.
In addition, from the fact that $f$ is $s$-self-involutive, it follows that  $\alpha^s$ is  also a root of  $f(x)$.
Thus  $\alpha$ and $\alpha^s$ are conjugated  and  there exists $ 1 \leq u \leq m-1$ such that $\alpha^s = \alpha^{q^u}$, or equivalently,  $s \equiv q^u \pmod {ord(\alpha)}$. 
Now, we observe that 
$$ (\alpha + \alpha^s) ^{q^u} = \alpha^{q^u} + (\alpha^{q^u})^s = \alpha^s + \alpha^{s^2} = \alpha^s +\alpha,$$
that implied that  $ \alpha^s +\alpha \in \F_{q^u} $. 
The same way
$$ (\alpha^{s+1})^{q^u} =\alpha^{s^2+s}= \alpha^{s+1}$$
and hence   $\alpha^{s+1} \in \F_{q^u} .$
It follows that
$$ \F_q(\alpha +\alpha^s,\alpha^{s+1}) \subseteq \F_{q^u} \cap \F_{q^m} = \F_{q^{\gcd(u,m)}} \subsetneq \F_{q^m} = \F_q(\alpha),$$
in particular we have that $[ \F_q(\alpha) : \F_q(\alpha +\alpha^s, \alpha^{s+1}) ] \ge 2$. 
On the other hand,  $\alpha$ is root of the polynomial $x^2 - (\alpha+\alpha^s)x + \alpha^{s+1}$, and this polynomial has its coefficients in  $\F_q(\alpha +\alpha^s,\alpha^{s+1}) $, therefore
$[ \F_q(\alpha) : \F_q(\alpha +\alpha^s, \alpha^{s+1}) ] \le 2$.
From these inequalities  we conclude that $[ \F_q(\alpha) : \F_q(\alpha +\alpha^s, \alpha^{s+1}) ] = 2$.\qed
\end{proof}

The following lemma let us understand how the towel of fields of the form $\F_q(\alpha^{2^k})$ grown, when we change the values of $k$. 

\begin{lemma}\label{Lema da extensao}
Let $\alpha\in \overline \F_q$ be a $2n$-th primitive root of the unit, where $n$ is even and  $q$ be a power of a prime such that  $q \equiv 3 \pmod 4$.  Then

\begin{enumerate}[(i)]
\item $ [ \F_q(\alpha^{2^{\nu_2(n)-1}}) : \F_q(\alpha^{2^{\nu_2(n)}})] = 2$  
and 
$ [ \F_q(\alpha^{2^{\nu_2(n)}}) : \F_q(\alpha^{2^{\nu_2(n)+1}})] = 1$.
\item If $\nu_2(n)\leq \nu_2(q+1),$ then
 $$ [ \F_q(\alpha^{2^{j}}) : \F_q(\alpha^{2^{j+1}})] = 1\quad  \hbox{for all } \quad 0 \le j \leq \nu_2(n)-2.$$
\item If $\nu_2(n) > \nu_2(q+1),$ then
$$ [ \F_q(\alpha^{2^{j}}) : \F_q(\alpha^{2^{j+1}})] = \begin{cases} 2 \quad\quad \text{if} & 0 \leq j < \nu_2(n)- \nu_2(q+1) \\
1 \quad\quad \text{if} & \nu_2(n)- \nu_2(q+1) \leq j \leq \nu_2(n) -2. \end{cases}$$
\end{enumerate}

\end{lemma}
\begin{proof}
Let us denote by $k=[ \F_q(\alpha) : \F_q] $, i.e., $k$ is  the degree of the minimal polynomial of   $\alpha$ over $\F_q$.  It is known that $k=ord_{2n} q$, i.e.,   the smallest positive integer $l$  such that
$q^l \equiv 1 \pmod{2n}$, in particular, we have that $$2 \leq \nu_2(2n) \leq \nu_2(q^k-1).$$
Since  $ q \equiv3 \pmod4$,  that inequality implies that $k$ is even.
For each $0 \leq j \leq \nu_2(n)$ we  denote
$ k_j=[ \F_q(\alpha^{2^j}) : \F_q ].$
It is clear that  $[ \F_q(\alpha^{2^{j-1}}) : \F_q(\alpha^{2^{j}})] = 1$  or $2$, and   hence either $ k_{j+1} = k_j$ or $k_{j+1} = 2k_j$.  
By definition,   $k_j$ is the smallest positive integer such  that the relation  $\alpha^{2^jq^{k_j}} = \alpha^{2^j}$ is satisfied,  or equivalently 
$\alpha^{2^j(q^{k_j}-1)} = 1$. In particular we have
 $$ {2^j(q^{k_j}-1)} \equiv 0 \pmod {2^{\nu_2(n)+1}}$$
and consequently 
 \begin{equation}\label{nu}
 \nu_2(q^{k_j}-1) \ge \nu_2(n)+1-j,\quad \text{for each $0 \leq j \leq \nu_2(n)$.}
 \end{equation}
  We observe  that in the case when 
  $j = \nu_2(n)$,  it follows that $\nu_2(q^{k_{\nu_2(n)}}-1) \ge 1$ and, from the minimality of  $k_{\nu_2(n)}$, we conclude that $k_{\nu_2(n)}$ is odd and $\nu_2(q^{k_{\nu_2(n)}}-1) = 1$.
Furthermore,  $[ \F_q(\alpha):\F_q]$ is even,  $[\F_q(\alpha^{2^{\nu_2(n)+1}}): \F_q]$ is odd and consequently  $k_{\nu_2(n)} = k_{\nu_2(n)+1}$.  
Putting  $j = \nu_2(n)-1$ in Inequality  (\ref{nu}) we obtain 
 $$ \nu_2(q^{k_{\nu_2(n)-1}}-1) \ge 2,$$
hence $k_{\nu_2(n)-1}$ is even and therefore $k_{\nu_2(n)-1} = 2k_{\nu_2(n)}$, i.e.
 $$[ \F_q(\alpha^{2^{\nu_2(n)-1}}) : \F_q(\alpha^{2^{\nu_2(n)}})] = 2.$$
 The following diagram  shows, partially,  the degrees  of intermediary field extensions  generated by the powers $\alpha^{2^i}$ of  $\alpha$.
$$\xymatrix{
\mathbb{F}_q(\alpha) \ar@{-}[d] \\ 
\mathbb{F}_q(\alpha^2) \ar@{-}[d] \\ 
\vdots \ar@{-}[d] \\ 
\mathbb{F}_q(\alpha^{2^{\nu_2(n)-1}}) \ar@{-}[d]\ar@{-}@/^1pc/[d]^2 \\ 
\mathbb{F}_q(\alpha^{2^{\nu_2(n)}}) \ar@{-}[d]\ar@{-}@/^1pc/[d]^1 \\ 
\mathbb{F}_q(\alpha^{2^{\nu_2(n)+1}}) \ar@{-}[d]\ar@{-}@/^1pc/[d]^{\text{odd}}\\ 
\mathbb{F}_q}$$
 
 We still need to analyze the fields on the top of this tower.  If  $0 \leq j < \nu_2(n)$, then $k_j$ is even. Rewriting  $k_j = 2k_j'$  and by Lemma \ref{LEL} we obtain
$$ \nu_2(q^{k_j}-1) = 1 + \nu_2(q+1) + \nu_2(k_j'),$$
thus   (\ref{nu})  is equivalent to
 \begin{equation}\label{nu2}\nu_2(k_j') \ge \nu_2(n)- \nu_2(q+1) -j,\end{equation}
where   $\nu_2(k_j')$ is  the smallest non-negative integer  satisfying this inequality. 
 We remark that  (\ref{nu2}) is trivial in the case when the right side of the inequality is less or equal to zero.  From here we need to consider two cases:
 \begin{enumerate}[1)]
 \item  If $\nu_2(n) \leq \nu_2(q+1)$ it follows that $\nu_2(k_j')=0$ and therefore
 $\nu_2(k_j)=1$ for every $0\leq j\leq \nu_2(n)-2$.
 Thus
  $$k_j = k_{j+1} \quad \hbox{and}\quad [\F_q(\alpha^{2^{j}}):\F_q(\alpha^{2^{j+1}})] = 1\quad\text{for every}\quad 0 \leq j \leq  \nu_2(n)-2.$$
 
 
 \item  If $\nu_2(n) > \nu_2(q+1)$, we have  to consider two possibilities:
\vspace{2mm}

 \begin{enumerate}[2.1)]
 \item If  $ j \ge \nu_2(n)-\nu_2(q+1)$, 
it follows that $\nu_2(k'_j)=1$ and therefore $k_j = k_{j+1}$  for every $j$ such that $\nu_2(n)-\nu_2(q+1) \leq j \leq \nu_2(n)-2$.
 \vspace{2mm}

 \item If $ j < \nu_2(n)-\nu_2(q+1)$
then 
$\nu_2(k'_j)=\nu_2(n)- \nu_2(q+1) -j $, hence 
$$\nu_2(k'_{j-1})=\nu_2(n)- \nu_2(q+1) -j+1 \quad\text{ and }\quad \nu_2(k'_j)= \nu_2(k'_{j-1})+1,$$  consequently  $[\F_q(\alpha^{2^{j-1}}):\F_q(\alpha^{2^{j}})] = 2$ for every $ 1\leq j < \nu_2(n)-\nu_2(q+1)$.  In particular,  $[\F_q(\alpha):\F_q(\alpha^2)] = 2$.\qed
 \end{enumerate}
 \end{enumerate}
\end{proof}
\begin{corollary}
Let $q$ be power of a prime such that $q \equiv 3 \pmod4$ and   $f_i$ be an  irreducible  $s$-seft-involutive  factor of $x^n-1 \in \F_q[x]$. Let us denote by $\xi_i$ some  root of  $f_i$.
\begin{enumerate}[1)]
\item If $\nu_2(n)>\nu_2(q+1)$ then
$$ \F_q(\xi_i+\xi_i^s, \xi_i^{s+1}) = \F_q(\xi_i^2).$$
\item If $\nu_2(n) \leq \nu_2(q+1)$ then
$$ \F_q(\xi_i+\xi_i^s, \xi_i^{s+1}) = \F_q(\xi_i^{2^{\nu_2(n)}}).$$

\end{enumerate}
\end{corollary}
\begin{proof}
The proof follows from previous  lemma  and  Lemma \ref{lema1}.\qed
\end{proof}

\section{Group algebra of split metacyclic group}

Throughout this section, $G$ is  a non-abelian  group with the following presentation
\begin{equation}\label{representationD} G = \langle{x,y \mid x^n= 1 = y^2 , xy =y  x^s \rangle}.
\end{equation}
The polynomial $x^n-1 \in \F_q[x]$ splits into monic irreducible factors as:
\begin{equation}\label{fatoracaoxn-1} x^n-1 = f_1f_2 \cdots f_rf_{r+1}f^{*_s}_{r+1}f_{r+2}f^{*_s}_{r+2} \cdots f_{r+t}f^{*_s}_{r+t},
\end{equation}
where $f_1 = x - 1$,  $ f_2 = x + 1$ if $n$ is even, $f^{*_s}_j = f_j $ for each  $2 \leq j \leq r$, where $r$ is the number of $s$-self-involutive factors  and $2t$  is the number of non $s$-self-involutive factors in the decomposition.

\begin{theorem} \label{teorema1}
Let $G$ be a metacyclic group defined by the relation  (\ref{representationD}), 
where $s^2 \equiv 1 \pmod n$.  If   $d := \gcd(n, s-1)$, then
$$ \F_qG  \cong   \displaystyle  \bigoplus_{\lll \mid d}2\cdot\frac{\phi(\lll)}{ord_{\lll}q}\F_q( \theta_{\lll})  \oplus  \bigoplus_{ {1 \leq i \leq r } \atop f_i(x) \nmid (x^d-1)} A_i  \oplus \bigoplus_{ {r+1 \leq i \leq r+t } \atop f_i(x) \nmid (x^d-1)} B_i ,$$
where $\theta_{\lll}$ is a  $\lll$-th primitive root of the unit  and 
$$
\begin{cases}
A_i \cong M_2(\F_q(\xi_i+\xi_i^s,\xi_i^{s+1}))  \\
B_i \cong M_2(\F_q(\xi_i))\\
\end{cases}$$
with $\xi_i$ is a root of  the polynomial  $f_i(x)$.
\end{theorem}

\begin{proof}
Let $H$ be a subgroup of $G$ defined as $H=\langle{z, y }\rangle$, where $z=x^{n/d}$. Since
$$zy= x^{n/d}y = y(x^{n/d})^s =y x^{n (s-1)/d} x^{n/d}= yx^{n/d}=yz,$$
then   $H$ is  an abelian group isomorphic to the abelianization of  $G$.
Let $\psi$ be  the group homomorphism defined by the generators of  $G$ as
$$\begin{array}{cccc}
\psi: & G & \longrightarrow & H \\
& x & \longmapsto & z \\
& y & \longmapsto & y .\end{array}$$
This homomorphism can  be  extended to a homomorphism of  group algebras
$$\Psi:  \F_qG  \longrightarrow  \F_qH,  $$
that is surjective  and then
 $\F_qH$ is  a non-simple component of   $\F_qG$.   Now, using that  $ H = \langle{y}\rangle \times \langle{z}\rangle$ and 
$ \F_q \langle{y}\rangle \cong \F_q \oplus \F_q,$
by  
(\ref{decom_ciclico}) we have that
\begin{align}\label{parteabeliana} \F_qH &= \F_q  \left(\langle{y}\rangle \times \langle{z}\rangle \right) \cong  \left(\F_q \oplus \F_q \right)\langle{z}\rangle \nonumber\\
&\cong \F_qC_d\oplus \F_qC_d 
\cong \bigoplus\limits_{\lll \mid d} 2\cdot\dfrac{\phi(\lll)}{ord_{\lll}q}\F_q( \theta_{\lll}),\end{align}
where the last  isomorphism follows from  Perlis-Walker's Theorem (Theorem \ref{perliswalker}).

This  isomorphism  can be  shown explicitly, considering  for each irreducible factor  $g(z)$ of  $z^d-1$ in $\F_q[z]$,   the algebra-homomorphism determined by the generators as
$$ \begin{array}{cccc} \psi_g: & \F_q H & \longrightarrow & \frac{\F_q[z]}{g(z)}\oplus\frac{\F_q[z]}{g(z)} \\
& z & \longmapsto & (\overline{z}, \overline{z}) \\
& y & \longmapsto & (1,  -1) \end{array}$$
%
%
%

We note that $\dfrac{\F_q[z]}{g(z)}\cong \F_q(\theta_{g}) $, where $\theta_{g}$ is a root of $g$ and therefore   $\theta_{g}$  is a  $d$-th root of  unit, not necessarily  primitive.  The isomorphism  is  the direct sum of the homomorphisms  constructed in the previous way. 

%
%
%
%
%

In the other hand, following the notation of 
(\ref{fatoracaoxn-1}),  let $f_i(x)$ be an irreducible factor of  $\frac{x^n-1}{x^d-1}$ and $\xi_i$ a root of $f_i$. 
Let us define $\tau_i$ the  homomorphism  determining by
$$ \begin{array}{cccc} \tau_i : & \F_q G & \longrightarrow & M_2(\F_q(\xi_i)) \\
& x & \longmapsto & \left( \begin{array}{cc} \xi_i & 0 \\ 0 & \xi_i^s \end{array}\right) \\
& y & \longmapsto & \left( \begin{array}{cc} 0 & 1 \\ 1 & 0 \end{array}\right).  \end{array} $$ 
By direct calculation it is easy to verify  that 
 $ ( \tau_i(x))^n = I$ and $\tau_i(x)\tau_i(y) = \tau_i(y)\tau_i(x)^s$ and then  $\tau_i$ is a well defined homomorphism.  These homomorphisms  are not necessarily  surjective.  Indeed, for each $1 \leq i \leq r$, let us define  $Z_i = \left( \begin{array}{cc} 1 & -\xi_i \\ 1 & -\xi_i^s \end{array}\right)$ and $\sigma_i$ be the map given by the conjugation   determined by  $Z_i$, i.e. 
\begin{equation}\label{comp_conj}
  \begin{array}{cccc} \sigma_i: & M_2(\F_q(\xi_i)) & \longrightarrow & M_2(\F_q(\xi_i)) \\ & X & \longmapsto& Z_i^{-1 } X Z_i \end{array}.
\end{equation}
Composing that homomorphism  with 
 $\tau_i$, we obtain a new homomorphism that satisfies  the relations 
$$ \sigma_i\circ\tau_i(x) = \left( \begin{array}{cc} 0 & \xi_i^{s+1} \\ -1 & \xi_i+\xi_i^s \end{array}\right) \quad \hbox{and} \quad \sigma_i\circ\tau_i(y) = \left( \begin{array}{cc} 1 & -(\xi_i+\xi_i^s) \\ 0 & -1 \end{array}\right),$$
consequently  the image  of the map  $\sigma_i\circ \tau_i$  is contained in  $\F_q( \xi_i+\xi_i^s , \xi_i^{s+1}) \subsetneq \F_q(\xi_i)$.
It follows that, for each  $1 < i \leq r$ 
\begin{equation}\label{dimensao}
 \dim_{\F_q}( Im(\tau_i)) = \dim_{\F_q}( Im(\sigma_i\circ \tau_i)) \leq 4\dim_{\F_q}( \F_q( \xi_i+\xi_i^s , \xi_i^{s+1})), \end{equation}
and 
\begin{equation}
\dim_{\F_q}( Im(\tau_i)) \leq 4 \dim_{\F_q}( \F_q( \xi_i)).\end{equation}
 in the case  when $r+1 \leq i \leq r+t$.

From the homomorphisms previously  defined,
let us define  $\tau$  the  $\F_q$-algebras  homomorphism
$$\displaystyle\bigoplus_{g|(x^d-1)} \psi_{g}\oplus \bigoplus_{f_i\mid \frac{x^n-1}{x^d-1}}\tau_i.$$ 
We claim that $\tau$ is injetive. Indeed, let $ u = P(x) + Q(x)y \in \F_qG$ be an element in $Ker(\tau)$, where $P(x)$ and $Q(x)$ are polynomials of degree less or equal to $n-1$.  Since $\tau(u)=0$, it follows that 
$\psi_g(u)=0$ for all  $g(x)|(x^d-1)$ and   $\tau_i(u)=0$  for each  $f_i\mid \frac {x^n-1}{x^d-1}.$ 
In the first case we have that 
$$ \psi_{g}(u) = \psi_{g}(P(x)+Q(x)y) = (P(\theta_{g}) + Q(\theta_{g}), P(\theta_{g})-Q(\theta_{g})) = (0,0),$$
where  $\theta_g$ is a root of $g(x)$. Whereas that the characteristic of $\F_q$ is different that 2, it follows that  $P(x)$ and $Q(x)$ are zero when we evaluate these polynomials  at the roots of  $x^d-1$. 
Besides that, for each $f_i\mid \frac{x^n-1}{x^d-1}$, we have that  
$$\tau_i(u) = \left( \begin{array}{cc} P(\xi_i) & Q(\xi_i) \\ Q(\xi_i^s) &P( \xi_i^s) \end{array}\right)  = \left( \begin{array}{cc} 0 & 0 \\ 0 & 0 \end{array}\right) $$
$$P( \xi_i^s) = P( \xi_i) = 0 \quad \hbox{and}\quad Q(\xi_i^s) = Q(\xi_i) = 0,$$
therefore $P(x)$ and $Q(x)$ are divisible by the polynomials  $f_i(x)$ and $f_i^{*_s}(x)$ and then also divisible by  $\frac {x^n-1}{x^d-1}$.   From these two results we obtain that  $P(x)$ and $Q(x)$ are divisible by  $x^n-1$ and since the degree of these polynomials are less that  $n$,  we conclude that $P(x)$ and $Q(x)$ are the null polynomial.
In conclusion,   the homomorphism
$$ \rho : F_qG \longrightarrow \bigoplus_ {\lll \mid d} 2\cdot\frac{\phi(\lll)}{ord_{\lll}q}\F_q( \theta_{\lll}) \oplus \bigoplus_ {i = 1}^{r} A_i  \oplus \bigoplus_{i = r+1}^{r+t} B_i$$
defined by 
$$  \begin{cases} \psi_{g}, & \text{if $g$ is an irreducible fator of  $x^d-1$} \\
\sigma_i\circ\tau_i, & \hbox{if} \quad 1 \leq i \leq r \quad \hbox{and}\quad f_i \nmid (x^d-1) \\
\tau_i, & \hbox{if}\quad r+1 \leq i \leq r+t \quad \hbox{and}\quad f_i \nmid (x^d-1)   \end{cases}$$
is an injective map. Finally by Lemma \ref{lema1} and  
(\ref{dimensao}) we have that 
\begin{align}
2n & \leq \dim_{\F_q}\displaystyle\Bigl(\bigoplus_{\lll \mid d}2\cdot\frac{\phi(\lll)}{ord_{\lll}q}\F_q( \theta_{\lll})\oplus\bigoplus_ {i = 1}^{r} A_i  \oplus \bigoplus_{i = r+1}^{r+t} B_i \Bigr) \nonumber\\
&\le
2d +  4\cdot \displaystyle\sum_{i=1}^{r}\dim_{\F_q}( \F_q( \xi_i+\xi_i^s , \xi_i^{s+1}))+ 4\cdot \sum_{i=r+1}^{r+t}\dim_{\F_q}( \F_q( \xi_i))  \nonumber\\
& = 2d + 2\cdot\displaystyle\sum_{i=1}^{r}deg(  f_i) + 2\cdot\sum_{i=r+1}^{r+t}2\cdot deg(f_i) 
= 2 \left( d + deg\left(\frac{x^n-1}{x^d-1}\right)\right) = 2n,\nonumber\end{align}
consequently  the homomorphism is also surjective. \qed
\end{proof}

\section{Group algebra of non-split metacyclic group}

Throughout this section, $G$ is  a group with the following presentation
\begin{equation}\label{representationG}
G = \langle{ x,y \mid x^{2n}= 1, y^2 = x^n , xy = yx^s \rangle},
\end{equation}
where $s^2\equiv 1\pmod {2n}$, 
and the polynomial $x^{2n}-1 \in \F_q[x]$ splits into monic irreducible factors as:
\begin{equation}\label{fatoracaox2n-1} x^{2n}-1 = f_1f_2 \cdots f_rf_{r+1}f^{*_s}_{r+1}f_{r+2}f^{*_s}_{r+2} \cdots f_{r+t}f^{*_s}_{r+t},
\end{equation}
where $f_1 = x - 1$,  $ f_2 = x + 1$  and $f^{*_s}_j = f_j $ for each  $2 \leq j \leq r$, where $r$ is the number of $s$-self-involutive factors  and $2t$  is the number of non $s$-self-involutive factors in the decomposition.
We observe that if $s=-1$ then $G$ is the generalized  quaternion group  $Q_n$. It knows that $\Q Q_n$ have some irreducible components that are isomorphic to  non-commutative division rings.  In contrast,  in $\F_qQ_n$ this type of component does not appear, because  Wedderburn's Little  Theorem (Theorem 2.55 in \cite{ LiNi}) guarantees that  every finite division ring is a field.

We observe that  $[x,y]=x^{-1}y^{-1}xy=x^{s-1}$  then the commutator group $[G, G]$  is generated by  $x^d$ where $d=\gcd(2n, s-1)$.   Therefore, the abelianization of  $G$ is isomorphic   to a subgroup $H$ of $G$ determined by
\begin{equation}\label{abelianizado}
G_{ab}\cong H:=
\langle z,y| z^d=1,\quad y^2=z^{d/2}, \quad zy=yz\rangle,\end{equation}
where $z=x^{2n/d}$.
As in the  previous section, let us define the group homomorphism 
$$\begin{array}{cccc}
\psi: & G & \longrightarrow & H \\
& x & \longmapsto & z \\
& y & \longmapsto & y,\end{array}$$
that is surjective and can be extend to a group-algebra homomorphism $\Psi:  \F_qG  \longrightarrow  \F_qH $ and  hence $ \F_qG\cong \F_qH\oplus Ker(\Psi)$. Thus, we  need to understand   the structure of $\F_qH$ and $Ker(\Psi)$. 
 In  following lemma, we show explicitly  the components of the abelian part. 

\begin{lemma}\label{parteabeliana}Let  $G$ be a group with presentation (\ref{representationG}) and $H$  be a subgroup of  $G$   that is isomorphic to  the abelianization of $G$ defined in (\ref{abelianizado}).
\begin{enumerate}[1)]
\item If $s\equiv 1\pmod4 $ then  $\F_qH \cong \displaystyle\bigoplus_ {\lll \mid d} 2\cdot\frac{\phi(\lll)}{ord_{\lll}q}\F_q(\theta_{\lll})$.
 
\item If  $s\equiv 3\pmod4 $ and $q\equiv 1\pmod 4$
then 
 $\F_qH \cong \displaystyle\bigoplus_ {\lll \mid {d/2}} 4\cdot\frac{\phi(\lll)}{ord_{\lll}q}\F_q(\theta_{\lll})$.

\item If  $s\equiv 3\pmod4 $ and $q\equiv 3\pmod 4$ then $\F_qH \cong
\displaystyle\bigoplus_ {\lll \mid {d/2}} 2\cdot\frac{\phi(\lll)}{ord_{\lll}q}\F_q(\theta_{\lll}) \oplus\frac{\phi(\lll)}{ord_{\lll}q^2}\F_q(\theta_{2\lll}) $.
\end{enumerate}
\end{lemma}
\begin{proof}
It follows from  Representation Theorem for  Abelian Groups  that  $H$  is a direct product of cyclic groups. In order to  show explicitly  the direct product, we consider three cases:

{\em Case 1)}
 $s \equiv 1 \pmod 4$:    
Since  $d = \gcd(2n, s-1)$, then  $\frac{d}{2}$ is even and $H$ has the following presentation:
$$ H = \langle{z,w \mid z^d = 1, w^2=1\rangle}\cong C_d\times C_2 \quad \text{where}\quad  w = yz^{d/4} .$$
From Perlis-Walker's Theorem we have that
$$\F_qH \cong (\F_q \oplus \F_q)C_d \cong \F_qC_d \oplus \F_qC_d \cong \bigoplus_ {\lll \mid d} 2\cdot\frac{\phi(\lll)}{ord_{\lll}q}\F_q(\theta_{\lll}).$$
This isomorphism can be presented explicitly as 
$$ \begin{array}{ccccc}
\psi_{\lll}:& \F_qH &\longrightarrow& \F_q(\theta_{\lll}) \oplus \F_q(\theta_{\lll})\\
&z&\longmapsto&(\theta_{\lll}, \theta_{\lll})\\
&w& \longmapsto&(1, -1)\end{array}$$
where, for each irreducible factor of   $x^d-1$, we  choose some root  $\theta_{\lll}$  of order $\lll$. There exist exactly $\frac{\phi(\lll)}{ord_{\lll}q}$ factors of that type.
From this isomorphism we can construct  a surjective homomorphism    $\F_qG\to \F_q(\theta_{\lll}) \oplus \F_q(\theta_{\lll})$, which by notation abuse  we also denote by  $\psi_{\lll}$, defined from the generators of $G$ as
$$ \begin{array}{ccccc}
\psi_{\lll}:& \F_qG &\longrightarrow& \F_q(\theta_{\lll}) \oplus \F_q(\theta_{\lll})\\
&x&\longmapsto&(\theta_{\lll}, \theta_{\lll})\\
&y& \longmapsto&(\theta_{\lll}^{d/4},-\theta_{\lll}^{d/4})\end{array}$$

 {\em Case 2)}
 $s \equiv 3\pmod4$ and  $q\equiv 1 \pmod 4$: 
In this case, $d$ is even but not divisible by 4.
We claim that  $H$ can be represented as
$$ H = \langle{t,y \mid t^{d/2} =1, y^4=1 \rangle},\quad \text{where}\quad  t = z^2 .$$
In order to  prove that affirmation,  it is enough to show that $z$ can be written  as an expression that depends of  $t$ and $y$,  indeed
$$ z = z^{d+1} = z^{d/2}z^{d/2+1} = y^2t^{\frac{d+2}{4}}.$$
Thus $H \cong C_{d/2} \times C_4$ and 
$$\F_qH \cong \F_q(C_4 \times C_{d/2}) \cong (\F_qC_4)C_{d/2} 
\cong (\F_q\oplus\F_q\oplus\F_q\oplus\F_q)C_{d/2} $$
$$\hspace{3.5cm}\cong \F_qC_{d/2}\oplus\F_qC_{d/2}\oplus\F_qC_{d/2}\oplus\F_qC_{d/2} .$$
Again  by Perlis Walker's Theorem  it follows that 
$$\F_qC_{d/2}\oplus\F_qC_{d/2}\oplus\F_qC_{d/2}\oplus\F_qC_{d/2} \cong \bigoplus_ {\lll \mid {d/2}} 4\cdot\frac{\phi(\lll)}{ord_{\lll}q}\F_q(\theta_{\lll}),$$
where the isomorphism can be explicitly determinated by
$$ \begin{array}{ccccc}
\psi_{\lll}:& \F_qH &\longrightarrow& \F_q(\theta_{\lll}) \oplus \F_q(\theta_{\lll}) \oplus \F_q(\theta_{\lll}) \oplus \F_q(\theta_{\lll})\\
&t&\longmapsto&(\theta_{\lll}, \theta_{\lll},\theta_{\lll}, \theta_{\lll})\\
&y& \longmapsto&(1,-1,\beta, -\beta)\end{array}$$
where,  for each irreducible factor of $x^{d/2}-1$,   we select  an arbitrary root $\theta_\lll$  of order $\lll$  and 
$\beta\in \F_q$ satisfies $\beta^2=-1$. 

The same way, this isomorphism generates a surjective homomorphism  $\F_qG\to  \F_q(\theta_{\lll}) \oplus \F_q(\theta_{\lll})$, 
defined by
$$ \begin{array}{ccccc}
\psi_{\lll}:& \F_qG &\longrightarrow& \F_q(\theta_{\lll}) \oplus \F_q(\theta_{\lll}) \oplus \F_q(\theta_{\lll}) \oplus \F_q(\theta_{\lll})\\
&t&\longmapsto&(\theta_{\lll}^{\frac{d+2}{4}}, \theta_{\lll}^{\frac{d+2}{4}},-\theta_{\lll}^{\frac{d+2}{4}}, -\theta_{\lll}^{\frac{d+2}{4}})\\
&y& \longmapsto&(1,-1,\beta, -\beta).\end{array}$$

{\em Case 3)}
 $s \equiv 3\pmod4$ and $q\equiv 3 \pmod 4$:
As at before case,  $H$  can be represented as
$$ H = \langle{t,y \mid t^{d/2} =1, y^4=1 \rangle},\quad \text{where}\quad  t = z^2 .$$
Since    $d/2$ is odd, then $\F_q(\theta)$ is  an  extension of odd degree for any $\frac{d}{2}$-th root of unit $\theta$, so  $-1$ is not a square  in $\F_q(\theta)$, consequently 
$$\F_q(C_4 \times C_{d/2}) \cong (\F_qC_4)C_2 \cong ( \F_q\oplus\F_q\oplus\F_{q^2})C_{d/2} \cong \F_qC_{d/2}\oplus\F_qC_{d/2}\oplus\F_{q^2}C_{d/2}.$$
$$ \hspace{2.5cm} \cong \displaystyle\bigoplus_ {\lll \mid {d/2}} 2\cdot\frac{\phi(\lll)}{ord_{\lll}q}\F_q(\theta_{\lll}) \oplus\frac{\phi(\lll)}{ord_{\lll}q^2}\F_q(\theta_{2\lll}) $$
where the isomorphism can be written explicitly as 
$$ \begin{array}{ccccc}
\psi_{\lll}:& \F_qH &\longrightarrow& \F_q(\theta_{\lll}) \oplus \F_q(\theta_{\lll}) \oplus \F_q(\theta_{\lll}, \beta) \\
&t&\longmapsto&(\theta_{\lll},\theta_{\lll}, \theta_{\lll})\\
&y& \longmapsto&(1,-1,\beta)\end{array}$$
where  $\theta_\lll$ is a root of $x^{d/2}-1$ of order $\lll$,  $\beta\in \F_{q^2}$ is a square root of $-1$. Therefore using  previous homomorphism  we  construct   a   homomorphism with domain  $\F_qG$  determined by the generators of $G$ as 
$$ \begin{array}{ccccc}
\psi_{\lll}:& \F_qG &\longrightarrow& \F_q(\theta_{\lll}) \oplus \F_q(\theta_{\lll}) \oplus \F_q(\theta_{\lll}, \beta) \\
&x&\longmapsto&(\theta_{\lll}^{\frac{d+2}{4}},\theta_{\lll}^{\frac{d+2}{4}},- \theta_{\lll}^{\frac{d+2}{4}})\\
&y& \longmapsto&(1,-1,\beta).\end{array}$$
\end{proof}

The components previously found
 correspond to the simple abelian components of the group algebra $\F_qG$. Thus, we now need to analyze the simple non-abelian  components of this group algebra, that we know  from the Weddenburn-Artin Theorem,  is direct sum of components that are isomorphic to  matrix  algebras  over some finite extension of $\F_q$.  
We will construct homomorphisms  from   
 $\F_qG$   to each of these components.

\begin{theorem} \label{teorema2}
Let  $G$ be a metacyclic group  with  presentation
$$ G = \langle{x,y \mid x^{2n}= 1, y^2 = x^n , xy =y x^s \rangle}$$
with  $s^{2} \equiv 1 \pmod {2n}$. Let  $d = \gcd(2n, s-1)$ and  $H$ be the subgroup $G$ defined in Lemma    \ref{parteabeliana}. 
Then  $\F_qG\cong \F_qH\oplus \mathcal L$, where  $\mathcal L$ can be written as  sum of simple components  of the form: 
$$ \mathcal L  \cong  \displaystyle  \bigoplus_{{1 \leq i \leq r } \atop f_i(x) \nmid (x^d-1)}A_i \oplus \bigoplus_{ {r+1 \leq i \leq r+t } \atop f_i(x) \nmid (x^d-1)} B_i ,$$
where  
$ A_i \cong M_2(\F_q(\xi_i+\xi_i^s,\xi_i^{s+1}))$, $B_i \cong M_2(\F_q(\xi_i)) $ and $\xi_i$ is a root of $f_i(x)$.
\end{theorem}

\begin{proof}
Since $\frac{1+x^n}{2}$ and  $\frac{1-x^n}{2}$  are central orthogonal idempotents of  $\F_qG$, we have that
$$ \F_qG \cong \F_qG \left(\frac{1+x^n}{2}\right) \oplus \F_qG \left(\frac{1-x^n}{2}\right).$$
In addition, in  $\F_qG \left(\frac{1+x^n}{2}\right)$ the element  $\overline{x}$ satisfies that  $\overline{x}^n = \overline{1},$ hence this component is isomorphic to  the group algebra generated by the group of  Theorem \ref{teorema1}.
Therefore, we only need to determine the non-abelian components of  $\F_qG \left(\frac{1-x^n}{2}\right)$ and these components are generated by the irreducible factors of $x^n+1$ such that do not divide $x^d-1$. 

Let  $f_i$ be an irreducible factor of $x^n+1$ such that $\gcd(f_i, x^d-1)=1$  and   $\xi_i$ be any root of $f_i$. 
Let consider the natural homomorphism generated by the representation of $G$ associated to  $\xi_i$, i.e. 
 $$ \begin{array}{cccc} \omega_i : & \F_q G & \longrightarrow & M_2(\F_q(\xi_i)) \\
& {x} & \longmapsto & \left( \begin{array}{cc} \xi_i & 0 \\ 0 & \xi_i^s \end{array}\right) \\
& {y} & \longmapsto & \left( \begin{array}{cc} 0 & -1 \\ 1 & 0 \end{array}\right)  \end{array} $$ 

By straightforward verification we know that  $\omega_i({x})^n = I$ and $ \omega_i({x})\omega_i({y}) = \omega_i({y})\omega_i({x})^s$, thereby  $\omega_i$ is a well defined homomorphism. In general, these homomorphisms are not surjective and we need to find an appropriate isomorphism such that composed with $\omega_i$ shows clearly what is the image of this composition. In fact,  for each $s$-self-involutive factor $f_i$   $1 \leq i \leq r$, let us define 
 $Z_i = \left( \begin{array}{cc} -\xi_i^s & \beta \\ \beta\xi_i & 1 \end{array}\right)$, where  $\beta$ is an element of  $\overline{\F}_q$  such that   $\beta^2 = -1$. Then 
$$  \begin{array}{cccc} \eta_i: & M_2(\F_q(\xi_i)) & \longrightarrow & M_2(\F_q(\xi_i,\beta)) \\ & X & \longmapsto& Z_i^{-1 } X Z_i \end{array}$$
is an automorphism such that 
$$ \eta_i\circ\omega_i(\overline{y}) = \left( \begin{array}{cc} -\beta & 0 \\ -(\xi_i + \xi_i^s) & \beta \end{array}\right) \quad \hbox{and } \quad \eta_i\circ\omega_i(\overline{x}) = \left( \begin{array}{cc} 0 & \beta \\ \xi_i^{s+1} & \xi_i+\xi_i^s \end{array}\right)$$
thus the image of the generators are in  $\F_q( \xi_i+\xi_i^s , \xi_i^{s+1}, \beta)$.
\\
We claim that, in the case that  $ q \equiv 1 \pmod 4,$ or $q \equiv 3 \pmod 4$ and $\nu_2(n) > \nu_2(q+1) $, the field  
$ \F_q( \xi_i+\xi_i^s, \xi_i^{s+1}, \beta)$ is a proper subfield of $\F_q(\xi_i)$. 
Indeed by Lemma  \ref{lema1}  we have that $$( \F_q(\xi_i) : \F_q(\xi_i +\xi_i^s, \xi_i^{s+1}) ) = 2.$$
In addition, if  $q \equiv 1 \pmod4$ then  $\beta \in \F_q$, or if  $q \equiv 3 \pmod 4$ and $\nu_2(n) > \nu_2(q+1)$ it follows from Lemma \ref{Lema da extensao}  that  $[\F_q(\xi_i) : \F_q] = l$  is divisible by 4.  Therefore $[ \F_q(\xi_i +\xi_i^s, \xi_i^{s+1}) : \F_q ] = \frac{l}{2}$ is even and $\beta \in \F_q(\xi_i +\xi_i^s, \xi_i^{s+1}).$
In any of these cases,  $\eta_i\circ\omega_i$ is an homomorphism such that  the image  is contained in  $M_2(\F_q(\xi_i+\xi_i^s, \xi_i^{s+1})).$

The last case to consider is when 
 $q \equiv 3 \pmod 4$ and $\nu_2(n) \leq \nu_2(q+1)$. Let  $f_i$ be a  $s$-self-involutive irreducible factor of $x^n+1$ such that  $f_i$ does not divide $x^d-1$ and  $\xi_i$ be a root of  $f_i$.
By Lemma  \ref{Lema da extensao} we know that 
$$ [ \F_q(\xi_i) : \F_q] = 2l \quad \text{with $l$ odd and }\quad [ \F_q(\xi_i^{2^{\nu_2(n)+1}}) : \F_q] = l.  $$
In addition, using the fact that $f_i$ is  $s$-self-involutive we also have that  
$$[ \F_q(\xi_i) : \F_q(\xi_i+\xi_i^s, \xi_i^{s+1})] = 2.$$
Thus
$$\F_{q^l} = \F_q(\xi_i^{\nu_2(n)+1}) = \F_q(\xi_i+\xi_i^s, \xi_i^{s+1})$$
and since 
$$(\xi_i^{s+1})^{(q^l-1)} = 1 \quad \text{and}\quad \nu_2(q^l-1) =  1$$
it follows that  $\nu_2(n) \leq \nu_2(s+1).$

Now we can suppose without loss of generality that  $\nu_2(s+1) \leq \nu_2(q+1)+1$,  because if necessary  we can change $s$ by  $s+2n$ and every conditions remain valid. In this case, there exists 
 $\theta_i \in \F_q(\xi_i) = \F_{q^{2l}}$ such that  $$\theta_i^{2^{\nu_2(s+1)-\nu_2(n)}} = \xi_i$$
and this element has order  $$2^{{\nu_2(s+1)-\nu_2(n)}}n = 2^{\nu_2(s+1)+1}m \quad \text{ with $m$ odd.}$$
Therefore $\theta$ is an element of $\F_q(\xi_i)$ that is not  a square in this field.
Moreover, from 
$$ s+1 \equiv 0 \pmod{2^{\nu_2(s+1)}} \quad \text{and}\quad q^l+1 \equiv 0 \pmod{2^{\nu_2(s+1)}}$$
we have 
$$ s \equiv q^l \pmod{2^{\nu_2(s+1)}} \quad \text{and}\quad s \equiv q^l \pmod {m2^{\nu_2(s+1)}}.$$
Consequently  $\theta_i$  is a root of a  $s$-self-involutive polynomial and 
$$ [\F_q(\theta_i) : \F_q(\theta_i + \theta_i^s, \theta_i^{s+1})] = 2.$$
Let $a, b$ be elements of  $\overline \F_q$ such that $a^2=-\theta_i$ and $b^2=\theta_i^s$ and $Z_i$ be the matrix defined as 
$$ Z_i = \left( \begin{array}{cc} a & b \\ -\xi_ia & -\xi_i^sb \end{array}\right). $$
Using this matrix, we can defined a conjugation over the image of $\omega_i$ determined by  the generators 
\begin{equation}\label{conj_X}
Z_i \cdot \left( \begin{array}{cc} \xi_i & 0 \\ 0 & \xi_i^s \end{array} \right) \cdot Z_i^{-1} = \left( \begin{array}{cc} 0 & -1 \\ \xi_i^{s+1} & \xi_i+\xi_i^{s} \end{array} \right)
\end{equation}
and
\begin{equation}\label{conj_Y}
Z_i \cdot \left( \begin{array}{cc} 0 &-1 \\ 1 & 0 \end{array} \right) \cdot Z_i^{-1} = \frac{1}{ab}\left( \begin{array}{cc} \dfrac{\theta_i\xi_i - (\theta_i\xi_i)^s}{\xi_i-\xi_i^s} &  \dfrac{\theta_i - \theta_i^s}{\xi_i-\xi_i^s} \\ \dfrac{(\xi_i^2\theta_i)^s-\xi_i^2\theta_i}{\xi_i-\xi_i^s} &  -\dfrac{\theta_i\xi_i - (\theta_i\xi_i)^s}{\xi_i-\xi_i^s}  \end{array} \right).  
\end{equation}
Clearly  the matrix in (\ref{conj_X}) has coefficients in $$\F_q(\theta_i+\theta_i^s, \theta_i^{s+1}) = \F_q(\xi_i+\xi_i^s, \xi_i^s) = \F_q(\xi_i^{2^{\nu(n)+1}}).$$
We claim that the  matrix in (\ref{conj_Y}) also has coefficients in this field.  Since  $(ab)^2 = -\theta_i^{s+1}$, in order to prove that $ab\in \F_{q^l}$  it is enough to show that  $-\theta_i^{s+1}$ is a square in that field. Indeed
$$(-\theta_i^{s+1})^{\frac{q^l-1}{2}} = -\theta_i^{(s+1)(\frac{q^l-1}{2})} \quad \text{and}\quad \nu_2\left((s+1)\left(\frac{q^l-1}{2}\right)\right) = \nu_2(s+1).$$
By  construction  $\theta_i$ has order $2^{\nu_2(s+1)+1}m$, thus $\theta_i^{(s+1)(\frac{q^l-1}{2})} \neq 1$.  Since $-1$  is not a square, therefore $-\theta_i^{s+1}$ is a square  in  $\F_q(\theta_i+\theta_i^s, \theta_i^{s+1}).$
Finally, each one of matrix entries  is of the form 
$$ G_{c,d}(w,z) :=\pm \frac{w^c-z^c}{w^d-z^d},$$
where  $w=\theta_i$, $z=\theta_i^s$  and  $c$ and  $d$ are appropriate integers.
Since $G_{c,d}(x,y) = G_{c,d}(y,x)\in \F_q(x,y)$, by the Fundamental Theorem of Symmetric Functions (see \cite[Theorem 2.20]{Hung}), there exists  $\widetilde{G}_{c,d}\in \F_q(x,y)$ quotient of polynomials such that $G_{c,d}(x,y) = \widetilde{G}_{c,d}(x+y, xy) $ and this identity proves that
each entry of the matrix is in the field $\F_q(\theta_i + \theta_i^s, \theta_i^{s+1})$.

To finish the proof, let us consider the homomorphism $\Lambda:\F_qG \rightarrow  \F_q H\oplus \mathcal L$ defined as the direct sum of the homomorphism
 $\F_qG \rightarrow \F_q H$  found     in Lemma \ref{parteabeliana}   with  the homomorphism $\F_qG \rightarrow \mathcal L$ found in Theorem  \ref{teorema2}, i.e., 
$$  \begin{cases} 
\eta_i \circ \omega_i & \text{if}\quad 1 \leq i \leq r \quad \hbox{and}\quad f_i \nmid x^d-1 \\
\omega_i & \text{if}\quad r+1 \leq i \leq r+t \quad \text{and}\quad f_i \nmid x^d-1 \end{cases}$$
Following the same steps of the proof of Theorem  \ref{teorema1}, it is easy to prove that if  $u = P(x)+Q(x)y$ is in the kernel of   $\Lambda$,  then  $P(x)$ and  $Q(x)$  are divisible by  $x^{2n}-1$, therefore   $P(x) \equiv 0$ e $Q(x) \equiv 0$  and consequently  $\Lambda$ is injection.  Calculating  the dimension of $\F_q H\oplus \mathcal L$ as a vector space over $\F_q$, we have
\begin{align*}
4n \leq &\dim_{\F_q}\Bigl(\displaystyle \F_qH \oplus\bigoplus_{{1 \leq i \leq r }\atop f_i(x)\nmid (x^d-1)}A_i  \oplus\bigoplus_{ {r+1 \leq i \leq r+t } \atop f_i(x) \nmid (x^d-1)}B_i\Bigr) \\
&
\leq 4d  + 4\cdot \hspace{-4mm}\sum_{i=1\atop f_i(x)\nmid (x^d-1)}^{r}\dim_{\F_q}( \F_q( \xi_i+\xi_i^s , \xi_i^{s+1})) 
 + 4\cdot \hspace{-4mm}\sum_{i=r+1\atop f_i(x)\nmid (x^d-1)}^{r+t}\dim_{\F_q}( \F_q( \xi_i)) 
\\
 &=  4d + 4\cdot\displaystyle\sum_{i=1\atop f_i(x)\nmid (x^d-1)}^{r}\frac{deg(  f_i)}{2} 
 + 4\cdot\displaystyle\sum_{i=r+1\atop f_i(x)\nmid (x^d-1)}^{r+t}deg(f_i) 
\\
& =  4 \left( d + deg\left(\frac{x^n+1}{x^d-1}\right)\right) = 4n, \end{align*}
which is  the dimension of $\F_qG$ over  $\F_q$, hence it proves that  $\Lambda$ is an isomorphism.\qed
%

\end{proof}

\begin{remark}
In order to determine the number of idempotents of $\F_qG$ we do not need to know the explicit factorization of the polynomial   $x^{2n}-1$ over  $\F_q[x]$, actually we only need to know how many simple component have 
 $\F_qG$ and to determine this number we observe that the polynomial  $x^{2n}-1$ can be factorize of the following form 
$$ x^{2n}-1 = \prod_{m\mid2n} \Phi_m(x),$$
where  $\Phi_m(x)$ is the cyclotomic polynomial of order $m$ and, from Theorem  2.47 in \cite{LiNi}, 
$\Phi_m(x)$ splits into $\phi(m)/d$ irreducible factors of degree $d=ord_m q$.  We observe that  the $s$-self-involutive  propriety depends only of the degree,  specifically  one factor  of order $m$ and degree $d$ is   $s$-self-involutive if and only if  $s \equiv q^{d/2} \pmod m$. \end{remark}

\section{Central Primitive Idempotents}
In this section, we show explicitly an expression  for every central primitive idempotent of the group algebra  $\F_qG$, where $G$ is a group that have presentation of the forms (\ref{representationD}) and (\ref{representationG}). Before we show the expression of the idempotents, we present  the form of the idempotents  when the group is cyclic. 

\begin{theorem}\cite[Theorem 2.1]{Bro2}\label{idempotente} 
Let  $\F_q$ be a finite field with $q$ elements and $n \in \N^*$ such that $\gcd(q,n)=1$, then each primitive idempotent of  $\F_qC_n$ is of the form:
$$ e_{d,j}(x) = \frac{x^n-1}{f_{d,j}(x)}h_{d,j}(x),$$
where $f_{d,j}$ is an irreducible factor of the cyclotomic polynomial  $\Phi_d(x)$, with $d$  a  divisor of $n$ and $h_{d,j}(x) \in \F_q[x]$  the polynomial of degree $\deg(h_{d,j}(x)) =
ord_dq$ that is the inverse of  $\frac{x^n-1}{f_{d,j}(x)}$ in the field $\frac{\F_q[x]}{\langle{f_{d,j}}\rangle}$.
\end{theorem}
We note that if we know the explicit form of the polynomial $f_{d,j}$, then the polynomial  $h_{d,j}$  can be calculate using the Extended Euclidean Algorithm for polynomials.

\begin{lemma}[\cite{Bro2}, Lemma 2.1] \label{Lemma_Id_Cn}
Let  $ \mathcal{I} \subset \frac{\F_q[x]}{\langle{x^n-1}\rangle}$ be an ideal generated by the monic polynomial  $g$ that is a  divisor of  $x^n-1$ and let  $ f = \frac{x^n-1}{g}$.  Then the 
 idempotent of  $\mathcal{I}$ is 
$$ e_f := -\frac{((f^*)')^*}{n}\cdot \frac{x^n-1}{f}$$
where  $f'(x)$ and  $f^*(x)$ denote  respectively the formal derivation and the reciprocal of the polynomial $f(x)$.
\end{lemma}
%
%
%
\begin{theorem}\label{qtdidempotente}
Let $G$  be a metacyclic group with the following presentation 
$$ G = \langle{x,y \mid x^n= 1 = y^2 , y^{-1}xy = x^s \rangle}$$
where $s^2 \equiv 1 \pmod n$ and let  $d = \gcd(n, s-1)$.
The central idempotents of the  group algebra  $\F_qG$ are of the following form:
\begin{enumerate}[(1)]
\item  For each irreducible divisor  $f_j$ of $x^d-1$, there are two primitive idempotents of the form            $\frac{1+y}{2}e_{f_j}$ and  $\frac{1-y}{2}e_{f_j}$.
\item For each $s$-self-involutive irreducible  factor $f_j$ of $ \frac{x^n-1}{x^d-1}$, there exists one primitive idempotent of the form  $e_{f_j}$.
\item For each pair of non $s$-self-involutive irreducible factors $f_j,f_j^{*_s}$ of   $\frac{x^n-1}{x^d-1}$, there exists one primitive idempotent of the form $e_{f_j} + e_{f_j^{*_s}}$.
\end{enumerate}
\end{theorem}

\begin{proof}
From Theorem  \ref{teorema1}, we know that the homomorphism 
$$ \rho : \F_qG \longrightarrow  \displaystyle  \bigoplus_{\gamma \mid d}2\cdot\frac{\phi(\gamma)}{ord_{\gamma}q}\F_q( \theta_{\gamma})  \oplus  \bigoplus_{ {1 \leq i \leq r } \atop f_i(x) \nmid (x^d-1)} A_i  \oplus \bigoplus_{ {r+1 \leq i \leq r+t } \atop f_i(x) \nmid (x^d-1)} B_i $$
is injective. 

Firstly we consider the idempotents of the non-abelian components, then  the image of a central primitive idempotent is an element of the form $ (0,0, \dots, I, 0, \dots, 0)$, i.e.  the image is zero on each component, except  in one where the image is the identity  $I$ of that component, i.e. for each primitive idempotent $e$ there exists $i$  such that   $\rho_j(e) = \delta_{i,j}I_j$ for all $j$,  where   $I_j$ is the identity of the component $j$ and $\delta_{i,j}=\begin{cases} 0, &\text{if $i\ne j$}\\ 1, &\text{if $i= j$}\end{cases}$.

Let $e = P(x) + Q(x)y$ be a presentation of $e$, where  $P(x)$ and $Q(x)$ are polynomials  in  $\F_q[x]$ of degree less or equal to $n-1$.  Since  $\rho(e) = I$, we have that $Q(\alpha_i) = 0$ for every  $\alpha_i$  root of $\dfrac{x^n-1}{x^d-1}$, hence  $Q(x)$ is divisible by this polynomial.   We claim that $Q(x)$ is also divisible by
$x^d-1$ and then it is  null polynomial.   At the   abelian components  we have that $P(\theta_j)+Q(\theta_j)=0$ and  $P(\theta_j)-Q(\theta_j)=0$  for any $\theta_j$ root of $x^d-1$, therefore $Q(x)$ is also divisible by $x^d-1$. 

Moreover, $P(x) = 1$  when we evaluate at the roots of the polynomials 
$f_j$ and $f_j^{*_s}$ and zero when we evaluate in any other root of   $x^n-1$.  From the injectivity of  $\rho$ and Lemma \ref{Lemma_Id_Cn} it follows that the unique polynomial of degree less or equal to $n-1$ 	that satisfies that propriety is  $e_{f_j}$ when $f_j = f_j^{*_s}$, and  $e_{f_j} + e_{f_j^{*_s}}$ when $f_j \neq f_j^{*_s}$.

Finally, if $f_i \mid (x^d-1)$ then the image of $\rho_j(e_{f_j}) = (1,1)$  is not a primitive idempotent. Therefore, that this idempotent can be decomposed in two central primitive  idempotents, $\frac{1+y}{2}e_{f_j}$ and  $\frac{1-y}{2}e_{f_j}$ such that  $\rho_j\left(\left(\frac{1+y}{2}\right)e_{f_j}\right)= (1,0)$ and  $\rho_j\left(\left(\frac{1-y}{2}\right)e_{f_j}\right)= (0,1)$.
\qed
\end{proof}

\begin{remark}
In the group determined by 
(\ref{representationG}),  the non-abelian idempotents  are of the same form as those found in previous theorem.  Now, in order to describe  the abelian idempotents, we need to consider some cases.

Let $f$ be an irreducible factor of  $x^d-1 \in \F_q[x]$.  The idempotent $e_f(x)$ is a non-primitive central  idempotent. The decomposition of these idempotents depend to the classes   of   $s$ and $q$ modulo 4. 

Using the notation of Lemma \ref{parteabeliana}, we have the following cases:
\begin{enumerate}[{Case} 1.]
\item $ s\equiv 1 \pmod 4$:  
Since  $H \cong C_d \times C_2$, then 
$$ e_f(x) = e_f(x)\left(\frac{1+w}{2}\right) + e_f(x)\left(\frac{1-w}{2}\right)\text{  where }w = x^{\frac{d}{4}}y.$$
In this case   $e_f(x)\left(\frac{1+x^{\frac{d}{4}}y}{2}\right)$ and $e_f(x)\left(\frac{1-x^{\frac{d}{4}}y}{2}\right)$ are central primitive idempotents.
\item $ s\equiv 3 \pmod 4$ e $q \equiv 1 \pmod 4$:
In this case   $H\cong C_{d/2} \times C_4$ and  the idempotent 
$e_f(x)$ can be decomposed as 
\begin{align*}
 e_f(x) = e_f(x)\left(\frac{1+y+y^2+y^3}{4}\right) +& e_f(x)\left(\frac{1-y+y^2-y^3}{4}\right) \\
 + &e_f(x)\left(\frac{1+\beta y-y^2-\beta y^3}{4}\right) + e_f(x)\left(\frac{1-\beta y-y^2+\beta y^3}{4}\right)
\end{align*}
where $\beta$ is an element of $\F_q$ such that $\beta^2 = -1$.
Using the fact that $x^n = y^2$, we can rewrite the  expression above as  
\begin{align*}
e_f(x) = e_f(x)\left(\frac{1+x^n}{2}\right)\left(\frac{1+y}{2}\right) + &e_f(x)\left(\frac{1+x^n}{2}\right)\left(\frac{1-y}{2}\right) \\
+& e_f(x)\left(\frac{1-x^n}{2}\right)\left(\frac{1+\beta y}{2}\right) + e_f(x)\left(\frac{1-x^n}{2}\right)\left(\frac{1-\beta y}{2}\right).
\end{align*}
Some of these terms  can be zero. Indeed,
\begin{enumerate}[$\bullet$]
\item If $f$ divides $x^n-1$ then $\displaystyle e_f(x) = e_f(x)\left(\frac{1+y}{2}\right)+ e_f(x) \left(\frac{1-y}{2}\right)$,
\item If  $f$ divides $x^n+1$ then $\displaystyle e_f(x) = e_f(x)\left(\frac{1+\beta y}{2}\right)+ e_f(x)\left(\frac{1-\beta y}{2}\right)$.
\end{enumerate}

\item $ s\equiv 3 \pmod 4$ and $ q \equiv 3 \pmod 4$: As in the previous case, we have that $H\cong C_{d/2} \times C_4$ but does not exists $\beta \in \F_q$ such that $\beta^2=-1$ . Therefore
 $e_f(x)$ can decomposed  as 
$$e_f(x) = e_f(x)\left(\frac{1+y+y^2+y^3}{4}\right) + e_f(x)\left(\frac{1-y+y^2-y^3}{4}\right) + e_f(x)\left(\frac{1- y^2}{2}\right) $$
Again using the relation $x^n = y^2$ we can simplify this expression:
\begin{enumerate}[$\bullet$]
\item If  $f$ divides $x^n-1$ then  $\displaystyle e_f(x) = e_f(x)\left(\frac{1+y}{2}\right)+ e_f(x) \left(\frac{1-y}{2}\right)$,
\item If $f$ divides $x^n+1$ then $e_f(x)$ is primitive.
\end{enumerate}

\end{enumerate}
\end{remark}
\section{ Non-central Idempotents}

Unlike what happens with central primitive idempotents, non-central idempotents are not necessarily unique.   In this section, we show a complete family of non-central idempotents finding, when  possible,  a decomposition of each central idempotent.  Moreover, we only need to find the idempotents generate by the factor of $x^n-1$ that are $s$-self-involutive, because in the case when $f$  is not $s$-self-involutive, we know that  $e_f$ and  $e_{f^*}$ are non-central primitive idempotents.
\begin{theorem}\label{idempotente}
Let  $G$ be a group with the following presentation
$$ G = \langle{x,y \mid x^{n}= 1, y^2=1 , xy = yx^s \rangle},$$
where $s^2 \equiv 1 \pmod {n}$ and $d = \gcd(n, s-1)$.
If  $f(x)$ is a $s$-self-involutive irreducible factor of  $\frac{x^{n}-1}{x^d-1}$ and $\ell(x)$  is a polynomial such that $(x^{s-1}-1)\ell(x) \equiv 1 \pmod {f(x)}$ then

$$
e_{f,1} =- \frac{1}{n}((f^{*})')^{*} \cdot \frac{x^{n}-1}{f(x)} \cdot [\ell(x)(1-y)+1]\text{  and } e_{f,2} = -\frac{1}{n}((f^{\ast})')^{\ast} \cdot \frac{x^{n}-1}{f(x)}\cdot [\ell(x)(1-y)]$$
are non-central primitive idempotents.
\end{theorem}

\begin{proof}
We observe that 
$$\gcd(x^{s-1}-1,f(x)) = \gcd( \gcd(x^{s-1}-1,x^n-1), f(x)) = \gcd(x^d-1,f(x)) = 1$$
and this relation ensures that there exists  a polynomial $\ell(x)$ that is  the inverse of $x^{s-1}-1$ modulo $f(x)$.

Let $u = P(x)+Q(x)y$ be a non-central idempotent such that its  image in the component generated by $f$  for  the isomorphism defined in Theorem \ref{teorema1} is of the form $\left(\begin{array}{cc} 1 & 0 \\ 0 & 0 \end{array}\right)$ and the image at the other components are zero.
If $\xi$ is a root of the $s$-self-involutive polynomial  $f(x)$, using  (\ref{comp_conj}) we have that 
$$ \sigma(u)=\left(\begin{array}{cc}
1 & -\xi \\ 1 & -\xi^s\end{array}\right)^{-1}
\left(\begin{array}{cc}
P(\xi) & Q(\xi) \\ Q(\xi^s) & P(\xi^s)\end{array}\right)
\left(\begin{array}{cc}
1 & -\xi \\ 1 & -\xi^s\end{array}\right) =\left(\begin{array}{cc}
1 & 0 \\ 0 & 0\end{array}\right) ,$$
thus
$$ \left(\begin{array}{cc}
P(\xi) & Q(\xi) \\ Q(\xi^s) & P(\xi^s)\end{array}\right)
 = \frac{1}{\xi - \xi^s}\left(\begin{array}{cc}
1 & -\xi \\ 1 & -\xi^s\end{array}\right)\left(\begin{array}{cc}
1 & 0 \\ 0 & 0\end{array}\right) \left(\begin{array}{cc}
-\xi^s & \xi \\ -1 & 1\end{array}\right)$$
and this way we get
\begin{equation}\label{6.1}
P(\xi) = \frac{-\xi^s}{\xi-\xi^s} = \frac{-\xi^{s-1}}{1 -\xi^{s-1}} , \qquad   
P(\xi^s) = \frac{\xi}{\xi-\xi^s} = \frac{1}{1 -\xi^{s-1}} 
\end{equation}
\begin{equation}\label{3}
 Q(\xi) = \frac{\xi}{\xi-\xi^s} = \frac{1}{1 -\xi^{s-1}},
\qquad
Q(\xi^s) = \frac{-\xi^s}{\xi-\xi^s} = \frac{-\xi^{s-1}}{1 -\xi^{s-1}}.
\end{equation}
We observe that  equations in (\ref{6.1})  are equivalents. Indeed
$$ P(\xi^s) = P(\xi^{q^{ord(\xi)/2}}) = ( P(\xi))^{ord(\xi)/2} = \left( \frac{-\xi^s}{\xi - \xi^s}\right)^{q^{ord(\xi)/2}} = \frac{\xi^{sq^{ord(\xi)/2}}}{\xi^{q^{ord(\xi)/2}}- \xi^{sq^{ord(\xi)/2}}}$$
$$ \hspace{-1cm} = \frac{-\xi^{s^2}}{\xi^s - \xi^{s^2}} = \frac{-\xi}{\xi^s-\xi} = \frac{\xi}{\xi - \xi^s}.$$
The same way  equations in  (\ref{3}) are equivalents.  In addition,  the image of $u$ at the other components are zero, it follows that 
$$ P(\theta) = 0 \text{ and }Q(\theta) = 0\text{ for every  }\theta \text{ root of }\frac{x^n-1}{f(x)}.$$
Denoting  by  $g(x)$ the polynomial $\frac{x^n-1}{f(x)}$, the relation above  implies that $g(x)$ is a divisor of $P(x)$ and $Q(x)$. Let us suppose that $P(x)=h(x)g(x)$, then from 
(\ref{6.1}) it follows that the polynomial $(x^{s-1}-1)P(x) - x^{s-1}$  is zero when we evaluate at $x=\xi$, and therefore  the irreducible polynomial $f(x)$ is a factor of this polynomial, i.e. 
$$ (x^{s-1} -1)P(x) \equiv x^{s-1} \pmod{f(x)}$$
which is equivalent to 
\begin{equation}\label{6.5}
(x^{s-1}-1)h(x)g(x) \equiv x^{s-1}\pmod{f(x)}.
\end{equation}
On the other hand, by Lemma \ref{Lemma_Id_Cn} we know that $-\frac{1}{n}((f^{*})')^{*}\cdot \frac{x^{n}-1}{f(x)}\equiv 1 \pmod{f(x)}$, hence multiplying  (\ref{6.5}) by $-\frac{1}{n}((f^{*})')^{*} $  we obtain
$$ (x^{s-1}-1)h(x) \equiv- \frac{1}{n}((f^{*})')^{*}x^{s-1}\pmod{f(x)},$$
and multiplying the previous equation by $\ell(x)$  it follows that
$$ h(x) = -\frac{1}{n}((f^{*})')^{*}x^{s-1}\ell(x) \equiv -\frac{1}{n}((f^{*})')^{*}(\ell(x)+1)\pmod{f(x)},$$
hence 
$$ P(x) \equiv  -\frac{1}{n}((f^{*})')^{*}\cdot \frac{x^{n}-1}{f(x)} \cdot(\ell(x)+1)\pmod{x^n-1}.$$ 
Following the same procedure for the polynomial $Q(x)$, we conclude that
$$Q(x)\equiv  \frac{1}{n}((f^{\ast})')^{\ast} \cdot \frac{x^{n}-1}{f(x)}\cdot \ell(x) \pmod{x^n-1},$$
and from these two relations we concluded that
$$e_{f,1} = -\frac{1}{n}((f^{*})')^{*} \cdot \frac{x^{n}-1}{f(x)} \cdot [\ell(x)(1-y)+1]$$
is a non-central primitive idempotent. To finish,  we observe that the orthogonal complement of $e_{f,1}$ can be easy calculating as
$$e_{f,2} = e_f - e_{f,1} =  -\frac{1}{n}((f^{\ast})')^{\ast} \cdot \frac{x^{n}-1}{f(x)}\cdot \ell(x)(1-y).$$
\qed
\end{proof}

The following theorem  describes partially a family of non-central idempotents  of the group algebra $\F_qG$, where $G$ is the group with presentation given by (\ref{representationG}). The same way as before, using the decomposition
$$ \F_qG \cong \F_qG\left(\frac{1+x^n}{2}\right) \oplus \F_qG\left(\frac{1-x^n}{2}\right),$$
we have that the idempotents of the first component have been described in previous theorem, so we are going  show the form of the idempotents of the second component.

\begin{theorem}\label{idempotente2}
Let $G$ be a group with the following presentation
$$ G = \langle{x,y \mid x^{2n}= 1, y^2=x^n , xy = yx^s \rangle}$$
where  $s^2 \equiv 1 \pmod {2n}$ and  $d = \gcd(2n, s-1)$.
Let  $f(x)$ be a $s$-self-involutive irreducible factor of  $x^n+1$ that does not divide  $x^d-1$  and  $\ell(x)$ is a polynomial such  that $(x^{s-1}-1)\ell(x) \equiv 1 \pmod {f(x)}$.  Then  a pair of orthogonal non-central primitive idempotents  $e_{f,1}$ and $e_{f.2}$, such that  $e_f=e_{f,1}+e_{f.2}$, are described in the following cases
\begin{itemize}
\item  If $q \equiv 1 \pmod 4$, then 
 $$e_{f,1} = -\frac{1}{2n}((f^{*})')^{*} \cdot \frac{x^{2n}-1}{f(x)} \cdot (\ell(x)+1)[1+\beta y]$$ 
and
 $$e_{f,2} = \frac{1}{2n}((f^{\ast})')^{\ast} \cdot \frac{x^{2n}-1}{f(x)}\cdot [h(x) + \beta(\ell(x)+1)y]$$ where  $\beta^2=-1.$
\item If $q \equiv 3 \pmod 4, \nu_2(n) > \nu_2(q+1)$ and  $s \equiv 1 \pmod 4$,
%
then 
$$e_{f,1} = -\frac{1}{2n}((f^{*})')^{*} \cdot \frac{x^{2n}-1}{f(x)} \cdot (\ell(x)+1)(1+x^{n/2}y)$$
and 
$$e_{f,2} = \frac{1}{2n}((f^{\ast})')^{\ast} \cdot \frac{x^{2n}-1}{f(x)}\cdot [\ell(x)+ x^{n/2}(\ell(x)+1)y)].$$

\item If $q \equiv 3 \pmod 4, \nu_2(n) \le  \nu_2(q+1)$, then
$$e_{f,1} = -\frac{1}{2n}((f^{*})')^{*} \cdot \frac{x^{2n}-1}{f(x)} \cdot \ell (x)f'(x) (-x^s+y).$$

\end{itemize}

\end{theorem}

\begin{proof}
The proof is essentially the same to previous  theorem. Indeed, if  $f(x)$ is an irreducible factor of $x^n+1$  does not divide $x^d-1$ and  $u = P(x) + Q(x)y$  is a non-central primitive idempotent  such that the projection on the component generated by $f$  is a matrix $\left(\begin{array}{cc} 1&0 \\ 0 &0 \end{array}\right)$, we have that 
$$ \left(\begin{array}{cc}
-\xi^s & \beta \\ \beta\xi & 1\end{array}\right)^{-1}
\left(\begin{array}{cc}
P(\xi) & Q(\xi) \\ Q(\xi^s) & P(\xi^s)\end{array}\right)
\left(\begin{array}{cc}
-\xi^s & \beta \\ \beta\xi & 1\end{array}\right) =\left(\begin{array}{cc}
1 & 0 \\ 0 & 0\end{array}\right), $$
where $\xi$ is any root of $f$ and $\beta\in \overline\F_q$ such that $\beta^2=-1$.  These relations are equivalent to
\begin{equation}\label{6.6} P(\xi) = \frac{-\xi^{s-1}}{1 -\xi^{s-1}}, \quad   P(\xi^s) = \frac{1}{1 -\xi^{s-1}} \end{equation}
\begin{equation}\label{6.7}Q(\xi)  = \frac{-\beta\xi^{s-1}}{1 -\xi^{s-1}}, \quad   Q(\xi^s) = \frac{\beta}{1 -\xi^{s-1}}\end{equation} 
We note that in the case when $q \equiv 1 \pmod 4$, the element $\beta$ is in 
$\F_q.$  It follows that $\xi$ and $\xi^s$ are  roots of the polynomials
$$ (x^{s-1}-1)P(x) - x^{s-1} \quad\text{  and  }\quad(x^{s-1}-1)Q(x)-\beta x^{s-1}.$$

From this point, the proof follows exactly with the  same steps  as the proof of previous theorem. 

In the case when  $q \equiv 3 \pmod 4$,  $-1$
 is not a square in $\F_q$ and therefore $\beta \notin \F_q$. Nevertheless  using that $\xi^n=-1$ and $n$ is even, we have that  $\beta = \xi^{n/2}$ and therefore
\begin{equation}\label{8}
Q(\xi) = \frac{-\xi^{n/2+s}}{\xi-\xi^s} \text{    and  }Q(\xi^s) = \frac{\xi^{n/2+1}}{\xi-\xi^s}.
\end{equation}
These two equations are equivalent  because 
$$ Q(\xi^s) = Q(\xi)^{q^{ord(\xi)/2}} = \frac{-\xi^{\frac{n}{2}s+s^2}}{\xi^s-\xi^{s^2}}= \frac{-\xi^{\frac{n}{2}s+1}}{\xi^s-\xi}=\frac{\xi^{n/2}}{1-\xi^{s-1}},$$
where in the last identity we use that $\xi^{n/2}$  is a quartic root of unite and  $s \equiv 1 \pmod 4$.
Thus  $\xi$ is a root of the polynomials 
$$ (x^{s-1}-1)P(x) - x^{s-1} \text{    and  }(x^{s-1}-1)Q(x) - x^{n/2+s-1}.$$

 If $q \equiv 3 \pmod 4$ and  $ \nu_2(n) \leq \nu_2(q+1)$, let    $u = P(x) + Q(x)y$  is a non-central primitive idempotent  such that de projection on the component generated by $f$  is a matrix $\left(\begin{array}{cc} 1&0 \\ 0 &0 \end{array}\right)$, we have that 
$$ \left(\begin{array}{cc}
a & b \\ -\xi a & -\xi^sb\end{array}\right)^{-1}
\left(\begin{array}{cc}
P(\xi) & Q(\xi) \\ Q(\xi^s) & P(\xi^s)\end{array}\right)
\left(\begin{array}{cc}
a & b \\ -\xi a & -\xi^s b\end{array}\right) =\left(\begin{array}{cc}
1 & 0 \\ 0 & 0\end{array}\right), $$
where $\xi$ is a root of $f(x)$ 
and $ a^2 = -\theta, b^2 = \theta^s$ according to the Theorem 4.2.  
These relations are equivalent to
\begin{equation}\label{6.6} P(\xi) = \frac{-\xi^{s-1}}{1 -\xi^{s-1}}, \quad   P(\xi^s) = \frac{1}{1 -\xi^{s-1}} \end{equation}
\begin{equation}\label{6.7}Q(\xi)  = \frac{1}{\xi-\xi^{s}}, \quad   Q(\xi^s) = \frac{\xi^s}{1 -\xi^{s-1}}.\end{equation}
  It follows that $\xi$ and $\xi^s$ are roots of the polynomials
$$ (x^{s-1}-1)P(x) -x^{s-1} \quad\text{  and  }\quad(x^{s}-x)Q(x)+1.$$

Let $P(x)$ and $Q(x)$ such that
$$ P(x) = -\frac 1{2n} \frac{x^{2n}-1}{f(x)}h(x) \quad \text{and}\quad Q(x) = -\frac 1{2n}\frac{x^{2n}-1}{f(x)}m(x)$$
where $h(x)$ is the polynomial 
$$ (1-x^{s-1})h(x) \equiv  x^s f^{'}(x) \pmod{ f(x)} \quad \text{and}\quad (1-x^{s})m(x) \equiv -f'(x)\pmod{f(x)}$$
 
Observe that the previous relation implies that  $h(x)\equiv -x^sm(x) \pmod{ f(x)}$.

At this point, the proof follows exactly  as the proof of previous theorem. 

$\qed$

\end{proof}

%
%

\end{document}